\documentclass[11pt]{amsart}

\usepackage{latexsym}
\usepackage{amsfonts}

\usepackage{amsmath,amsthm,amssymb}

\usepackage{color}

\usepackage[headings]{fullpage}

\author[I.~Kapovich]{Ilya Kapovich}

\address{\tt Department of Mathematics, University of Illinois at
  Urbana-Champaign, 1409 West Green Street, Urbana, IL 61801, USA
  \newline http://www.math.uiuc.edu/\~{}kapovich/} \email{\tt
  kapovich@math.uiuc.edu}

\title[Detecting fully irreducible automorphisms]{Detecting fully
  irreducible automorphisms: a polynomial time algorithm. \emph{With an appendix by Mark C. Bell.}}

\newtheorem{thm}{Theorem}[section] \newtheorem{lem}[thm]{Lemma}
\newtheorem{cor}[thm]{Corollary} 
\newtheorem{prop}[thm]{Proposition} \theoremstyle{definition}
\newtheorem{defn}[thm]{Definition}

\newtheorem{conv}[thm]{Convention} \newtheorem{rem}[thm]{Remark}

\newtheorem{propdfn}[thm]{Proposition-Definition}

\def\strutdepth{\dp\strutbox}
\def \ss{\strut\vadjust{\kern-\strutdepth \sss}}
\def \sss{\vtop to \strutdepth{
\baselineskip\strutdepth\vss\llap{$\diamondsuit\;\;$}\null}}

\def\strutdepth{\dp\strutbox}
\def \sst{\strut\vadjust{\kern-\strutdepth \ssss}}
\def \ssss{\vtop to \strutdepth{
\baselineskip\strutdepth\vss\llap{$\spadesuit\;\;$}\null}}

\def\strutdepth{\dp\strutbox}
\def \ssh{\strut\vadjust{\kern-\strutdepth \sssh}}
\def \sssh{\vtop to \strutdepth{
\baselineskip\strutdepth\vss\llap{$\heartsuit\;\;$}\null}}

\def\epsilon{\varepsilon}
\def\phi{\varphi}
\def\Pr{\mathbb P}
\def\hat{\widehat}

\newcommand{\Curr}{\mbox{Curr}}

\DeclareMathOperator{\Out}{Out}
\newcommand{\Aut}{\mbox{Aut}}

\newcommand{\FN}{F_N}   

\newcommand{\CVNbar}{\overline{\mbox{CV}}_N}

\newcommand{\PCurr}{\Pr\Curr(\FN)}

\usepackage{amssymb, amsmath, amscd, amsfonts, amsthm}

\newcommand{\inlineQED}{\pushQED{\qed} \qedhere \popQED}

\newcommand{\NN}{\mathbb{N}}
\usepackage{colonequals}
\newcommand{\defeq}{\colonequals}
\DeclareMathOperator{\poly}{poly}

\DeclareMathOperator{\PF}{PF}
\newcommand{\lamPF}{\lambda_{\textrm{PF}}}

\usepackage{url}

\newenvironment{absolutelynopagebreak}
  {\par\nobreak\vfil\penalty0\vfilneg
     \vtop\bgroup}
  {\par\xdef\tpd{\the\prevdepth}\egroup
     \prevdepth=\tpd}

\begin{document}

\begin{abstract}
In \cite{Ka14} we produced an algorithm for deciding whether or not an
element $\phi\in \Out(F_N)$ is an iwip (``fully irreducible")
automorphism. At several points that algorithm was rather inefficient
as it involved some general enumeration procedures as well as running
several abstract processes in parallel. In this paper we refine the
algorithm from \cite{Ka14} by eliminating these inefficient features,
and also by eliminating any use of mapping class groups
algorithms. 

Our main result is to produce, for any fixed $N\ge 3$, an algorithm which, given a topological representative $f$ of an element $\phi$ of $\Out(F_N)$, decides in polynomial time  in terms of the ``size'' of $f$, whether or not $\phi$ is fully irreducible.

In addition, we provide a train track criterion of being fully
irreducible which covers all fully irreducible elements of
$\Out(F_N)$, including both atoroidal and non-atoroidal ones. 

We also give an algorithm, alternative to that of Turner, for finding all the
indivisible Nielsen paths in an expanding train track map, and
estimate the complexity of this algorithm. 

An appendix by Mark Bell
provides a polynomial upper bound, in terms of the size of the
topological representative, on the complexity of the Bestvina-Handel
algorithm\cite{BH92} for finding either an irreducible train track
representative or a topological reduction.

\end{abstract}

\thanks{The author was supported by the NSF
  grant DMS-1405146}


\maketitle

\section{Introduction}

For an integer   $N\ge 2$ an outer automorphism $\phi\in\Out(F_N)$ is called \emph{fully irreducible} if there no positive power of $\phi$ preserves the conjugacy class of any proper free factor of $F_N$.  The notion of being fully irreducible serves as the main $\Out(F_N)$ counterpart of the notion of a pseudo-Anosov mapping class of finite type surface. Fully irreducible automorphisms are crucial in the study of of dynamics and geometry of $\Out(F_N)$, and understanding the structural properties of its elements and of its subgroups.  If $\phi\in\Out(F_N)$ is fully irreducible, then $\phi$ acts with with North-South dynamics on the compactified Outer space $\CVNbar$~\cite{LL}, and, under the extra assumption of $\phi$ being atoroidal, $\phi$ also acts with North-South dynamics on the projectivized space of geodesic currents $\PCurr$~\cite{Mar} (see also \cite{U}).  By a result of Bestvina and Feighn~\cite{BF14}, $\phi\in \Out(F_N)$ acts as a loxodromic isometry of the free factor complex of $F_N$ if and only if  $\phi$ is fully irreducible. The case where a subgroup $H\le \Out(F_N)$ contains some fully irreducible is the first case to be treated in the solution of the Tits Alternative for $\Out(F_N)$, see~\cite{BFH97,BFH00,BFH05}. The modern subgroup structure theory of $\Out(F_N)$ also involves this setting. Thus a result of Handel and Mosher~\cite{HM} shows that if $H\le \Out(F_N)$ is a finitely generated subgroup then either $H$ contains a fully irreducible element or $H$ contains a subgroup $H_0$ of finite index in $H$, such that $H_0$ leaves invariant a proper free factor of $F_N$. Subsequently, Horbez~\cite{Hor} proved, by a very different  argument, that this  result remains true with the assumption that $H$ be finitely generated dropped.  See \cite{CH12,Gui1,KL4} for other examples illustrating the important role that fully irreducibles play in the study of $\Out(F_N)$. 

In \cite{Ka14} we produced an algorithm that, given $\phi\in\Out(F_N)$, decided whether or not $\phi$ is fully irreducible. That algorithm did not include a complexity estimate and, as written, was computationally quite inefficient.  There were several places in the algorithm were two general enumeration procedures were run in parallel, and in one subcase the algorithm also invoked an algorithm from surface theory as a ``black box". 

In the present paper we remedy this situation and produce a refined algorithm for deciding whether or not an element $\phi$ of $\Out(F_N)$ is fully irreducible. This improved algorithm runs in polynomial time in terms of the ``size" of $\phi$ or of a topological representative of $\phi$. The algorithm also does not use any surface theory algorithms as subroutines.  The main result is stated in Theorem~\ref{thm:A} below. The term ``standard" for a topological representative $f'$ of $\phi$ is defined in Section~\ref{sec:b}, but basically it just means that we view $f'$ as a combinatorial object for algorithmic purposes, as is traditionally done in train track literature.  

Given a finite graph $\Delta$, and graph map $g:\Delta\to\Delta$, we
denote $||g||:=\max_{e\in E\Delta} |g(e)|$, where $|g(e)|$ is the combinatorial length of the edge-path $g(e)$. 
Let $N\ge 2$ and let $A=\{a_1,\dots, a_N\}$ be a free basis of $F_N$. Let $\Phi\in \Aut(F_N)$. Denote by $|\Phi|_A:=\max_{i=1}^N |\Phi(a_i)|_A$. 

The main result  of this paper is:

\begin{thm}\label{thm:A}
Let $N\ge 2$ be fixed.

\begin{enumerate}

\item There exists an algorithm that, given a standard topological
  representative $f':\Gamma'\to \Gamma'$ of some $\phi\in \Out(F_N)$,
  such that every vertex in $\Gamma$ has degree $\ge 3$, decides
  whether or not $\phi$ is fully irreducible.  The algorithm terminates in polynomial time in terms of $||f'||$. 

\item Let $A=\{a_1,\dots,a_N\}$ be a fixed free basis of $F_N$. There exists an algorithm, that given $\Phi\in \Aut(F_N)$ (where $\Phi$ is given as an $N$-tuple of freely reduced words over $A^{\pm 1}$, $(\Phi(a_1),\dots, \Phi(a_N))$), decides, in polynomial time in $|\Phi|_A$, whether or not $\Phi$ is fully irreducible.

\item Let $S$ be a finite generating set for $\Out(F_N)$. There exists
  an algorithm that, given a word $w$ of length
  $n$ over $S^{\pm 1}$, decides whether or not the element $\phi$ of
  $\Out(F_N)$ represented by $w$ is fully irreducible. This algorithm terminates in at most exponential time in terms of the length
of the word $w$.

\end{enumerate}

\end{thm} 

After our paper \cite{Ka14}, Clay, Mangahas and Pettet~\cite{CMP} produced a different algorithm for deciding whether an element of $\phi\in \Out(F_N)$ is fully irreducible. Their algorithm was more efficient than that in \cite{Ka14}, but did not include an explicit complexity estimate. A key step in their approach was to reduce to a subcase where one can bound the size of the Stallings subgroup graph of a periodic free factor $B$ of $\Phi$ by a polynomial in terms of $|\Phi|_A$. They then enumerate \emph{all} Stallings subgroup graphs of possible free factors with that size bound and check if any of them are $\phi$-periodic. This approach produces, at best an exponential time algorithm in terms of the size of $\phi$ or of a topological representative of $\phi$ (in the setting of parts (1) and (2) of Theorem~\ref{thm:A} above), and at  best a double exponential algorithm in terms of the word length of $\phi$ in the generators of $\Out(F_N)$ (in the setting of part (3) Theorem~\ref{thm:A}).
Feighn and Handel also gave an alternate algorithm in \cite{FH} for deciding if an element of $\Out(F_N)$ is fully irreducible. They use their general extremely powerful machinery of CT train tracks and algorithmic computability of CT train tracks for rotationless elements of $\Out(F_N)$.  Their proof does not come with a complexity estimate and it remains unclear if a reasonable complexity estimate can be extracted from that proof. 

A key fact that goes into the proof of Theorem~\ref{thm:A} is Theorem~\ref{thm:clean} providing a train track characterization of being fully irreducible that covers all elements of $\Out(F_N)$  (both atoroidal and non-atoroidal ones).  See \cite[Lemma~9.9]{Pf}, \cite[Proposition~4.1]{Pf15}, \cite[Proposition~5.1]{JL},  \cite[Proposition~4.4]{Ka14} for related earlier train track criteria of being fully irreducible, in more restricted settings.  We include the statement of Theorem~\ref{thm:clean} here because of its potential further utility:

\begin{thm}\label{thm:clean'}
Let $N\ge 2$ and let $\phi\in \Out(F_N)$ be arbitrary.
Then the following are equivalent:

\begin{enumerate}
\item The automorphism $\phi$ is fully irreducible.
\item The automorphism $\phi$ is primitively atoroidal and there exists a weakly clean train track representative $f:\Gamma\to\Gamma$ of $\phi$.
\item The automorphism $\phi$ is primitively atoroidal, there exists a weakly clean train track representative of $\phi$, and every train track representative of $\phi$ is weakly clean.
\end{enumerate}
 
\end{thm}
The term \emph{weakly clean train track representative} is formally defined in Section~\ref{Sec:clean}. An element $\phi\in \Out(F_N)$ is called \emph{primitively atoroidal} if there do not exist $n\ne 0$ and a primitive element $g\in F_N$ such that $\phi^n[g]=[g]$.  This notion, or rather its negation, first appeared in \cite{CMP}, where non primitively atoroidal elements of $\Out(F_N)$ are called \emph{cyclically reducible}. Primitively atoroidal elements of $\Out(F_N)$ are also specifically studied in \cite{FH}. 

As a step in the proof of Theorem~\ref{thm:A} we produce an algorithm, with a polynomial time complexity estimate, that given an expanding standard train track map $f$, finds all the INPs (if any)  for $f$; see Theorem~\ref{thm:alg-inp}.  We then refine this algorithm to produce a corresponding statement for finding all periodic INPs, see Theorem~\ref{thm:turner} below.
The first (and until now the only) known algorithm for finding INPs in an expanding train track map was due to Turner~\cite{Turner}, and did not contain a complexity estimate.\footnote{There are some issues with Turner's result requiring further clarification. The ``expanding" assumption is missing from the statement of his theorem, and the category of graph maps in which he works is not sufficiently precisely defined, including the precise meaning of a ``train track map" of what exactly it means to ``compute" and INP and its endpoints algorithmically.} 
Our algorithm provided here is based on fairly different considerations from those of Turner. The main idea is to use precise quantitative estimates of bounded cancellation constants to get efficient control of the length of the two ``legs" of an INP.  This approach turns out to be sufficient to obtain a polynomial time algorithm for finding INPs and periodic INPs.

The algorithm given in Theorem~\ref{thm:A} needs another component, namely the classic Bestvina-Handel algorithm~\cite{BH92} for trying to find a train track representative of a free group automorphism. The Bestvina-Handel algorithm is one of the most useful results in the study of $\Out(F_N)$, both as a theoretical and as an experimental tool. Yet, amazingly enough, until now the computational complexity of this algorithm has not been analyzed. The appendix to the present paper, written by Mark Bell, fills that gap.  Theorem~\ref{thm:BH92} of the appendix provides a polynomial time estimate, in terms of the size $||f||$ of a given topological representative of an element of $\Out(F_N)$ (where $N\ge 2$ is fixed), on the running time of the Bestvina-Handel algorithm.

The most complicated part of the algorithm given in Theorem~\ref{thm:A} consists in finding all the pINPs in a expanding irreducible train-track representative $f:\Gamma\to\Gamma$ for $\phi$ produced by the Bestvina-Handel algorithm, and then processing this data to decide whether a) $\phi$ is not primitively atoroidal (in which case we conclude that $\phi$ is not fully irreducible and the algorithm stops), or b) $\phi$ is not fully irreducible for some other reason, or c) $\phi$ is primitively atoroidal (in which case we then use Theorem~\ref{thm:clean'} above). 
It is now known, for geometric reasons, that under fairly mild assumptions on the support of a measure defining a random walk, long random walks on $\Out(F_N)$ produce elements that are atoroidal and fully irreducible, with probability tending to $1$ as the length of the walk tends to $\infty$. See \cite{MT} as the main general result about random walks on groups acting on hyperbolic spaces, together with \cite{DT} for the relevant action of $\Out(F_N)$ on the ``co-surface graph". It is also known~\cite{KP} that a special type of a random walk on $\Out(F_N)$ (where $N\ge 3$), with asymptotically positive probability produces an element of $\Out(F_N)$  which is atoroidal, fully irreducible and is represented by an expanding irreducible train track map with no pINPs. Moreover, experimental data using Thierry Coulbois' train track package~\cite{Col} suggests that a simple random walk on $\Out(F_N)$  generically produces an automorphisms that is represented by an expanded irreducible train track map with no pINPs. 
Thus it is likely that for processing ``random" inputs, the more complicated parts of the algorithm in Theorem~\ref{thm:A} dealing with manipulating pINPs are rarely needed. It would be extremely interesting to obtain an actual  proof of the above suggested absence of pINPs for automorphisms generically produced by  a simple random walk on  $\Out(F_N)$. In particular, such a result would shed light on the index properties of random elements of $\Out(F_N)$ (see \cite{CH12} for a survey of the $\Out(F_N)$ index theory). 

We thank the referees of this paper for helpful comments. We are particularly grateful to one of the referees for pointing out a simpler and more efficient proof of Proposition~\ref{prop:pa} below than our original argument. Moreover, our original proof of that proposition did not include a complexity estimate, while the argument suggested by the referee produced a polynomial time complexity estimate in the conclusion of Proposition~\ref{prop:pa}. 
This change also resulted in a simplification of the proof of our main result, Theorem~\ref{thm:A}.

\section{Background}\label{sec:b}

We adopt the same definitions, terminology, and notations regarding graph maps and
train track maps as in \cite{Ka14} and \cite{DKL}. Although this point
is not made explicit in \cite{Ka14} (because it is not directly
relevant there), we now explicitly assume that all the graphs under
considerations are equipped with a PL structure (that is, in the
terminology of \cite{DKL} they are \emph{linear graphs}) and that all graph maps and
train track maps under considerations are PL-maps, or, more precisely,
that in the terminology of \cite{DKL} they are \emph{linear graph
  maps}. Section~2 of \cite{DKL} discusses in detail the notion of a
linear graph map and a more general notion of a ``topological graph
map'', where for the the latter one does not impose any restrictions
on the nature of homeomorphisms by which subdivision intervals in an
edge get mapped to the edges of the graph. For many combinatorial
considerations in train track theory the distinction between
topological graph maps and linear graph maps is immaterial. This
distinction is almost never discussed in the train track literature. The
various papers on this subject almost never explain what they actually
mean by saying
that a graph map maps an edge to an edge-path (although they typically
implicitly assume that they are working in the PL category, that is,
they are working with linear graph maps). The distinction between
topological graph maps and linear graph maps does matter when Nielsen
paths and periodic points of train track maps are discussed. That
is why we adopt the convention of working only with linear graph maps
here. (Indeed, in the category of topological graph maps an expanding
irreducible train track map may have uncountably many periodic and
fixed points and may exhibit other pathologies). We refer the reader
to Section~2 of \cite{DKL} for an extended discussion on the topic.

\begin{defn}[Standard graph map]
We say that a graph map $f:\Gamma\to\Gamma$ is \emph{standard} if
$\Gamma$ is equipped with a linear structure identifying every open
edge $e\in E\Gamma$ with an interval $(0,1)$ via a map $\xi_e: e\to (0,1)$ with the following property. If $f(e)=e_1\dots e_n$,
then for each open subinterval $J_i$ of $e$ that $f$ maps to $e_i$
(where $i=1,\dots, n$) we have
$\xi_e(J_i)=(\frac{i-1}{n},\frac{i}{n})$ and, moreover, $\xi_{e_i}\circ
f\circ \xi_e^{-1}$ maps $(\frac{i-1}{n},\frac{i}{n})$ by an affine
orientation-preserving bijection to the interval $(0,1)$. 
\end{defn}

\begin{conv}
We will always assume (unless explicitly stated otherwise) that various graph
maps are linear graph maps. Most abstract definitions and results will be stated in that context. 

For algorithmic results we will assume that our graph maps are
standard graph maps, and this assumption will always be mentioned explicitly.
This assumption is not burdensome since every
graph map $f:\Gamma\to\Gamma$ is isotopic relative the vertices
of $\Gamma$ to a standard graph map $g:\Gamma\to\Gamma$. Note however, that in this situation it need not
be the case that $f^n:\Gamma\to\Gamma$ is standard (where $n\ge
1$). Therefore some care will need to be taken when working with
powers of standard graph maps, but for most combinatorial considerations these
issues will be immaterial.
\end{conv}

We also adopt the definitions of \cite{DKL} regarding
the notion of a \emph{path} in a graph $\Gamma$. We assume that all
paths under considerations are PL-paths, in the terminology of Section
2 of \cite{DKL}. In particular, every path $\gamma$ in a graph $\Gamma$ has an
initial point $o(\gamma)$ and a terminal point $t(\gamma)$. The points
$o(\gamma)$, $t(\gamma)$ are not required to be vertices of
$\Gamma$. For a path $\gamma$ in $\Gamma$ we denote by $\gamma^{-1}$
the inverse path of $\gamma$, so that $o(\gamma^{-1})=t(\gamma)$ and $t(\gamma^{-1})=o(\gamma)$.
A path $\gamma$ is \emph{nontrivial} if $o(\gamma)\ne t(\gamma)$.  Note that an essential closed path in $\Gamma$ is necessarily nontrivial.

For an edge-path $\gamma$ in a graph $\Gamma$ we denote by $|\gamma|$
the combinatorial length, that is the number of edges, in $\gamma$.

For any path $\alpha$ in a graph $\Gamma$ there exists a
unique smallest edge-path $\hat\alpha$ containing $\alpha$ as a
subpath. We call $\hat\alpha$ the \emph{simplicial support} of
$\alpha$ and denote $|\alpha|:=|\hat\alpha|$.

All of the train track  maps  in this paper are assumed to be homotopy
equivalences (although it turns out that dropping this requirement
sometimes does lead to useful results, see~\cite{AR,DKL1,DKL2,Rey}.

Recall that a train track map $f:\Gamma\to\Gamma$ is called
\emph{irreducible} if the transition matrix $M(f)$ is irreducible,
that is, if for every position $i,j$ there exists $n\ge 1$
such that $(M(f)^n)_{ij}>0$. Equivalently, a train track map
$f:\Gamma\to\Gamma$ is irreducible if and only if for every (oriented)
edges $e,e'$ of $\Gamma$ there exists $n\ge 1$ such that the path
$f^n(e)$ contains an occurrence of $e'$ or of $(e')^{-1}$.  Recall
also that a graph map $f:\Gamma\to\Gamma$ is called
\emph{expanding} if for every edge $e$ of $\Gamma$ the combinatorial
length $|f^n(e)|$ of the path $f^n(e)$ goes to infinity as
$n\to\infty$. For an irreducible train-track map $f:\Gamma\to\Gamma$
being expanding is equivalent to the condition $\lambda(f)>1$, where
$\lambda(f)$ is the Perron-Frobenius eigenvalue of the matrix $M(f)$.

We recall the following key notions from train track theory:

\begin{defn}[Nielsen paths]

Let $f:\Gamma\to\Gamma$ be an expanding train map.

\begin{enumerate}
\item A point $x\in \Gamma$ is \emph{$f$-periodic} or just
  \emph{periodic} if there exists $n\ge 1$ such that $f^n(x)=x$. For a
  periodic point $x\in \Gamma$, the
  smallest $n\ge 1$ such that $f^n(x)=x$ is called the \emph{period}
  of $x$.
\item A nontrivial tight path $\gamma$ in $\Gamma$ is called a \emph{Nielsen path} for
  $f$ if $f(\gamma)$ is homotopic to $\gamma$ relative the endpoints
  of $\gamma$ (which implies that the endpoints of $\gamma$ are fixed
  by $f$). 
\item A nontrivial tight path $\gamma$ in $\Gamma$ is called a \emph{periodic Nielsen
    path} for $f$ if there exists $n\ge 1$ such that $\gamma$ is a
  Nielsen path for $f^n$; in this case the smallest such $n\ge 1$ is
  called the \emph{period} of $\gamma$. (Thus if $\gamma$ is a
  periodic Nielsen path of period $n$ then each of the endpoints of
  $\gamma$ is periodic of period $\le n$ and is fixed by $f^n$).
\item A \emph{periodic indivisible Nielsen path} (or pINP) is a
  periodic Nielsen path which cannot be written as a concatenation of two
  nontrivial periodic Nielsen paths.
\item An \emph{indivisible Nielsen path} (or INP) is a
  Nielsen path which cannot be written as a concatenation of two
  nontrivial Nielsen paths.
\end{enumerate}

\end{defn}

We recall the following key facts (see \cite{BH92,BFH97} for background info):

\begin{prop}\label{prop:key}
Let $f:\Gamma\to\Gamma$ be an expanding train track map. Then the following holds:

\begin{enumerate}
\item There are only finitely many (possibly none) periodic
  indivisible Nielsen paths in $\Gamma$ for $f$.
\item If $\eta$ is a pINP, then $\eta$ has the form
  $\eta=\alpha\beta^{-1}$, where $\alpha,\beta$ are nontrivial legal
  paths with $v=t(\alpha)=t(\beta)\in V\Gamma$ (but where
  $o(\alpha),o(\beta)$ need not be vertices of $\gamma$) such that the turn at
  $v$ between $\alpha$ and $\beta$ is illegal.
\item A path $\eta$ is a pINP of period $1$ for $f$ if and only if
  $\eta$ is an INP for $f$.
\item For two pINPs $\eta_1,\eta_2$ with $t(\eta_2)=o(\eta_1)$ the
  following conditions are equivalent:
\begin{enumerate}
\item[(a)] the path $\eta_1\eta_2$ is tight;
\item[(b)] the path $\eta_1\eta_2$ is legal.
\end{enumerate}
\item If $\eta_1=\alpha_1\beta_1^{-1}$ and
  $\eta_2=\alpha_2\beta_2^{-1}$ are two pINPs written in the form
  given by part (2) of this lemma and such that
  $o(\alpha_1)=o(\alpha_2)=v\in V\Gamma$ and that directions at $v$
  given by $\alpha_1$ and $\alpha_2$ are the same, then $\alpha_1$ is
  an initial segment of $\alpha_2$ or $\alpha_2$ is an initial segment
  of $\alpha_1$. 
\item For a nontrivial cyclically tight circuit $\gamma$ in $\Gamma$, $[\gamma]$
  represents an $f$-periodic conjugacy class in $\pi_1(\Gamma)$ if and
  only if $\gamma$ has the form $\gamma=\eta_1\dots \dots \eta_k$,
  where $k\ge 1$, where each $\eta_i$ is a pINP and where the
  concatenation $\eta_i\eta_{i+1}$ is tight for $i=1,\dots, k$ (with
  $\eta_{k+1}$ interpreted as $\eta_1$).
\end{enumerate}

\end{prop}

We will also need the following key algorithmic result essentially due
to Turner~\cite{Turner}:

\begin{prop}\label{prop:turner}
There exists an algorithm that, given an expanding standard train
track map $f:\Gamma\to\Gamma$, finds all the periodic indivisible
Nielsen paths for $f$. 
\end{prop}

Later we will show (see Theorem~\ref{thm:turner} below) that one can provide a reasonably satisfactory complexity
estimate for the above statement.

\begin{proof}
By a result of Feighn and Handel~\cite{FH}, there exists a computable universal bound
$n=n(b_1(\Gamma))\ge 1$ such that $\eta$ is a periodic indivisible
Nielsen path for $f$ if and only if $\eta$ is an INP for $f^n$. 

We then apply the result of Turner~\cite{Turner} to algorithmically
find all the INPs for $f^n$.
\end{proof}

\begin{rem}\label{rem:FH}
In fact, Feighn and Handel show in \cite[Corollary 3.14]{FH} that
one can take $n(r)=3^{r^2-1}\left(g(15r-15)\right)!$, where $r=b_1(\Gamma)$. Here
$g(k)$ is Landau's function, which computes the maximal order of an
element in the symmetric group $S_k$. It is known that $g(k)\le
e^{k/e}$, so that $g(k)$ grows at most exponentially in $k$. 
\end{rem}

\section{Graph of periodic Nielsen paths}

We need the following useful definition, which is a slightly adapted version
of the graph $S(f)$ defined in Section~8.2 of \cite{FH} in the context
of CT maps. 

In this section we will assume that $\Gamma$ is a finite connected graph with all vertices of degree $\ge 3$ and with $b_1(\Gamma)=N\ge2$, and equipped  with a fixed marking identifying $\pi_1(\Gamma)$ with $F_N$.

\begin{defn}[Graph of periodic Nielsen paths]
Let $f:\Gamma\to\Gamma$ be an expanding irreducible train track map.

We define the \emph{graph of periodic Nielsen paths} $S(f)$ for $f$ as
follows. 

If there are no pINPs for $f$, define the graph $S(f)$ to be empty.

Otherwise, enumerate all the distinct periodic indivisible Nielsen paths
$\eta_1,\dots, \eta_k$ for $f$, where we eliminate duplication due to
inversion (that is, for a pINP $\eta$, we include only one of the
pINPs $\eta,\eta^{-1}$ in the list). Let $Y_f$ be the set of all $x\in
\Gamma$ such that $x$ occurs as an endpoint of some $\eta_i$. Put $V(S(f)):=Y_f$. 

For each $i=1,\dots, k$ we put a topological edge (i.e. a 1-cell) in
$S(f)$ joining vertices $o(\eta_i)$ and $t(\eta_i)$ of $S(f)$. We
orient this edge from $o(\eta_i)$ to $t(\eta_i)$ and label it by
$\eta_i$.

Thus $S(f)$ is a 1-complex as well as an oriented labelled graph in
the sense of \cite{Ka14}.

The edge-labels in $S(f)$ naturally define, by using concatenation, the \emph{labelling map}
$\mu$, associating to every edge-path (respectively, every circuit) in
$S(f)$ an edge-path  (respectively a circuit) in $\Gamma$.
\end{defn}

If $x\in Y_f$, then the labelling map $\mu$ defines a natural
homomorphism $\mu_{*,x}:\pi_1(S_f,x)\to \pi_1(\Gamma,x)$. Note that
this homomorphism, a priori, need not be injective, since it is
possible that two distinct pINPs start with the same vertex and the
same direction in $\Gamma$. So the image of a tight circuit at $x$ in
$S_f$ under the labelling map $\mu$ may not be tight as a circuit in $\Gamma$.

Note that any cyclic concatenation of pINPs that tightens to a
homotopically nontrivial closed path in $\Gamma$ defines a nontrivial
$f_\#$-periodic conjugacy class in $\pi_1(\Gamma)$. 
Proposition~\ref{prop:key} says that every $f_\#$-periodic conjugacy
class arises in this way. Therefore Proposition~\ref{prop:key} implies:

\begin{prop}\label{prop:mu}
Let $f:\Gamma\to\Gamma$ be an expanding irreducible train track map. 

Then a conjugacy class $[g]$ in $\pi_1(\Gamma)$ is $f_\#$periodic if and only if there is $x\in Y_f$ such that $[g]$ is represented by an element of $\mu_{*,x}(\pi_1(S(f),x))$.

In particular, $f_\#$ is atoroidal if and only if for every $x\in Y_f$ we have
$\mu_{*,x}(\pi_1(S(f),x))=\{1\}$.

\end{prop}

We also record the following useful property of the graph $S(f)$ which follows directly from Proposition~\ref{prop:key}.

\begin{prop}\label{prop:comp}
Let $f$ be as in Proposition~\ref{prop:key}.

Let $Q$ be a connected component of $S(f)$, let $x$ be a vertex of $Q$ and let $U\le F_N$ be the image of $\pi_1(Q,x)$ in $F_N$ under $\mu_{*,x}$. Then there is some $n\ge 1$ such that for the element $\phi\in \Out(F_N)$ represented by $f$ we have $\phi^n[U]=[U]$, and, moreover, there is a representative of $\phi^n$ in $Aut(F_N)$ which restricts to the identity map on $U$. 
\end{prop}

\section{Clean train tracks and full irreducibility}\label{Sec:clean}

\begin{defn}[Whitehead graphs and stable laminations]\label{defn:wh-lam}
Let $f:\Gamma\to\Gamma$ be an expanding
irreducible train track map.
Recall for a vertex $v\in V\Gamma$, the  \emph{Whitehead
  graph} $Wh_\Gamma(v,f)$ is a simple graph that has as its
vertices the set of all $e\in E\Gamma$ with $o(e)=v$, and where two
distinct such edges $e,e'\in E\Gamma$ are adjacent in
$Wh_\Gamma(v,f)$ if and only if there exist $e''\in E\Gamma$ and $n\ge
1$ such that $e^{-1}e'$ is a subpath of $f^n(e'')$ (note that in this
case $(e')^{-1}e$ is a subpath of $f^n((e'')^{-1})$, so that this
adjacency condition is symmetric). Recall also that the \emph{stable
  lamination} $\Lambda(f)$ of $f$ consists of all bi-infinite tight paths (called \emph{leaves} of $\Lambda(f)$)
$\gamma= \dots, e_{-1}, e_0,e_1,e_2,\dots $ in $\Gamma$ such that for
all $-\infty<i<j<\infty$ the path $e_i\dots e_j$ occurs as a subpath
of $f^n(e)$ for some $e\in E\Gamma$ and some $n\ge 1$. See \cite{Ka14}
for more details.
\end{defn}

We also recall the following key notion from \cite{Ka14}:

\begin{defn}[Clean train track map]
Let $f:\Gamma\to\Gamma$ be a train track map.

\begin{enumerate}
\item We say that $f$ is \emph{clean} if the matrix $M(f)$ is
  primitive (that is, there exists $t\ge 1$ such that every entry of
  $M(f^t)=(M(f))^t$ is positive) and if for every vertex $v\in
  V\Gamma$ the Whitehead graph $Wh_\Gamma(v,f)$ is connected.
\item We say that $f$ is \emph{weakly clean} if $f$ is expanding and
  irreducible (that is,  the matrix $M(f)$ is irreducible and has $\lambda(f)>1$) and if for every vertex $v\in
  V\Gamma$ the Whitehead graph $Wh_\Gamma(v,f)$ is connected.
\end{enumerate}
\end{defn}

Although this fact is not obvious from the definition, \cite{DKL}
showed:

\begin{prop}\label{prop:DKL}
Let $f:\Gamma\to\Gamma$ be an expanding irreducible train track map. Then $f$ is clean if and only if $f$ is weakly clean.
\end{prop}

We recall the following key result from \cite{Ka14}:
\begin{prop}\label{prop:carry}
Let $f:\Gamma\to\Gamma$ be a clean expanding irreducible train track map.

Let $1\ne H\le \pi_1(\Gamma)$ be a finitely generated subgroup of infinite index in $\pi_1(\Gamma)$ and let $\Delta_H$ be the Stallings core of the cover of $\Gamma$ corresponding to $H$. Then $\Delta_H$ does not carry any leaf of $\Lambda(f)$, that is, no leaf of $\Lambda(f)$ lifts to a path in $\Delta_H$.
\end{prop}

Recall that an automorphism $\phi\in\Out(F_N)$ is called
\emph{atoroidal} if there do not exist $n\ge 1$ and $1\ne g\in F_N$
such that $\phi^n([g])=[g]$, where $[g]$ is the conjugacy class of $g$
in $F_N$. We need the following slightly more general notion (see
Definition~13.1 of Feighn and Handel~\cite{FH}):

\begin{defn}
We say that $\phi\in\Out(F_N)$ is \emph{primitively atoroidal} if there do not exist $n\ge 1$ and a primitive element $g\in F_N$ such that $\phi^n([g])=[g]$.
\end{defn}

Recall the following well-known result of Bestvina-Handel~\cite{BH92}:
\begin{prop}\label{prop:BH92}
Let $\phi\in \Out(F_N)$ be non-atoroidal, where $N\ge 2$.

Then $\phi$ is fully irreducible if and only if there exist a compact connected (possible non-orientable) surface $\Sigma$ with one boundary component and with $\pi_1(\Sigma)=F_N$ and a pseudo-Anosov homeomorphism $g:\Sigma\to\Sigma$ such that $\phi=g_\#$.  In this case $\phi$ is primitively atoroidal, and $[w]$ is a nontrivial $\phi$-periodic conjugacy class in $F_N$ if and only if $w=[u^k]$ for some $k\ne 0$, where 
$u$ is the peripheral curve of $\Sigma$. 
\end{prop}

\begin{rem}\label{rem:BH92}
Thus all fully irreducible elements of $\Out(F_N)$ are primitively
atoroidal. Moreover, if $\phi\in\Out(F_N)$ admits periodic conjugacy
classes $[w], [z]$ such that for all nonzero $t,s$ we have $[w^t]\ne
[z^s]$, then $\phi$ is not fully irreducible. 
\end{rem}

The following result provides a unified train track characterization
of all fully irreducibles, both atoroidal and non-atoroidal ones,
following our approach in \cite{Ka14}.

\begin{thm}\label{thm:clean}
Let $N\ge 2$ and let $\phi\in \Out(F_N)$ be arbitrary.
Then the following are equivalent:

\begin{enumerate}
\item The automorphism $\phi$ is fully irreducible.
\item The automorphism $\phi$ is primitively atoroidal and there exists a weakly clean train track representative $f:\Gamma\to\Gamma$ of $\phi$.
\item The automorphism $\phi$ is primitively atoroidal, there exists a weakly clean train track representative of $\phi$, and every train track representative of $\phi$ is weakly clean.
\end{enumerate}
 
\end{thm}
\begin{proof}

We first show that (1) implies (3). Thus suppose that $\phi$ is fully
irreducible. Note that by Remark~\ref{rem:BH92} $\phi$ is primitively
atoroidal. Also, by a result of Bestvina-Handel~\cite{BH92}, there exists an expanding irreducible train track representative of $\phi$.
Let $f:\Gamma\to\Gamma$ be any train track representative of
$\phi$. Since $\phi$ is fully irreducible, exactly the same argument
as in the proof of implication (1)$\Rightarrow$(2) of Proposition~4.4 of \cite{Ka14} shows that $f$ is weakly clean.  Thus (1) indeed implies (3).

Part (3) directly implies part (2). 

Thus suppose that (2) holds and that $f:\Gamma\to\Gamma$ is a weakly
clean train track representative of a primitively atoroidal
automorphism $\phi\in \Out(F_N)$.  Then $f$ is clean by Proposition~\ref{prop:DKL}. Suppose that $\phi$ is not fully irreducible. Let $H\le F_N$ be a proper free factor of minimal rank such that the conjugacy class $[H]$ of $H$ is $\phi$-periodic. Then $rank(H)\ge 2$ since the case $rank(H)=1$ is ruled out  by the assumption that $\phi$ is primitively atoroidal. The minimality assumption on $H$ implies that there exists a representative $\Psi\in \Aut(F_N)$ of $\phi^k$ for some $k\ge 1$ such that $\Psi(H)=H$ and such that $\Psi|_H$  is a fully irreducible automorphism of $H$.  Hence there exists $1\ne h\in H$ such that the conjugacy class $[h]$ of $h$ is not $\phi$-periodic.
Let $\Delta_H$ be the $\Gamma$-Stallings core for $H$.  Let $\gamma$ be an immersed circuit in $\Gamma$ representing the conjugacy class of $h$. Since $[h]$ is not $\phi$-periodic, the cyclically tightened length of $f^n(\gamma)$ tends to $\infty$ as $n\to\infty$. Let $s$ be the simplicial length of $\gamma$, so that $\gamma=e_1\dots e_s$. 
 Let $\gamma_n$ be the immersed circuit in $\Gamma$ given by the cyclically tightened form of $f^{n}(\gamma)$. We can obtain $\gamma_n$ by cyclic tightening of the path $f^{n}(e_1)\dots f^{n}(e_s)$. Thus $\gamma_n$ is a concatenation of $\le s$ segments, each of which is a subsegment of $f^{n}(e)$ for some $e\in E\Gamma$. Since the simplicial length of $\gamma_n$ goes to infinity as $n\to\infty$, the length of at least one of these segments tends to infinity as $n\to\infty$. Since $\phi[H]=[H]$, the circuit $\gamma_n$ lifts to a circuit in $\Delta_H$. Hence there exists a sequence of segments $\alpha_n$ in $\Gamma$ such that each $\alpha_n$ lifts to a path in $\Delta_H$, such that the simplicial length of $\alpha_n$ goes to infinity as $n\to\infty$ and such that there are $e_n\in E\Gamma$ and  $t_n\ge 1$ with the property that $\alpha_n$ is a subpath of $f^{t_n}(e_n)$. Moreover, since $E\Gamma$ is finite, after passing to a subsequence we can even assume that $e_n=e\in E\Gamma$ for all $n\ge 1$. By a standard compactness argument, it follows that $H$ carries a leaf of $\Lambda(f)$, contrary to the conclusion of Proposition~\ref{prop:carry}. 
\end{proof}

\section{Finding INPs}

In this section we give an alternative algorithm (to that of Turner~\cite{Turner})
for finding all INPs in an expanding train track map. Our algorithm
is based on more direct and elementary considerations than that of
Turner, although it is possible that Turner's algorithm is
computationally more efficient. 

The algorithm presented here is known in the folklore. Essential ideas for this algorithm, particularly regarding bounding the length of the ``legs'' of an INP using the bounded cancellation constant, are already present in the 1992 paper of Bestvina and Handel~\cite{BH92}.  The structure of the algorithm is similar to that of the algorithm for finding INPs in a CT train track obtained by Feighn and Handel in \cite{FH}. However, the computational complexity analysis that we provide here, is new, and this analysis plays an essential role in the proof of our main result, Theorem~\ref{thm:A}.

For a graph map $g:\Delta\to\Delta$ (where $\Delta$ is a finite graph), we denote $||g||:=\max_{e\in
  E\Delta} |g(e)|$. Also, for a finite graph $\Delta$ we denote by
$m(\Delta)$ the number of topological edges of $\Gamma$.

Note that if $\Delta$ is a finite connected graph with all vertices of
degree $\ge 3$ and if $r$ is the rank
of the free group $\pi_1(\Delta)$ then $r\le m(\Delta)\le 3r-3$.

\begin{thm}\label{thm:alg-inp}
The exists a deterministic algorithm $\mathfrak A$ with the following property.

Given an expanding standard train track map $f:\Gamma\to\Gamma$ (where $\Gamma$
is a finite connected graph with all vertices of degree $\ge 3$), the
algorithm $\mathfrak A$ determines whether or not $f$ has any INPs and if yes, finds all the INPs for $f$. The algorithm terminates in
time $O(m^5||f||^{3m+6}\log m\log ||f||)$, where $m=m(\Gamma)$. In particular, if $m(\Gamma)$ is fixed, the algorithm runs in polynomial time in $||f||$. 
\end{thm}

\begin{conv}\label{conv:Gamma}
For the remainder of this section we will assume that
$f:\Gamma\to\Gamma$ is an expanding standard train track map, where $\Gamma$ is
a finite graph where every vertex has degree $\ge 3$. We denote
$m:=m(\Gamma)$.

\end{conv}

\begin{rem}
In the setting of Theorem~\ref{thm:alg-inp} the number $m$ of
topological edges of $\Gamma$ is not fixed, but rather is a variable
parameter of the problem. Therefore, in various algorithmic
computations below, we first need
  to index the edges of $\Gamma$ by binary integers, which introduces
  a factor of $\log m$ in various computations. In fact, describing
  the graph map $f:\Gamma\to\Gamma$ itself requires time $O(||f||m\log
  m)$ for that reason. Similar considerations apply to the subdivision
  $\Gamma'$ of $\Gamma$ discussed below.
\end{rem}

Recall that for a path (not necessarily an edge-path) $\gamma$ in some graph,
we denote by $|\gamma|$ the combinatorial length of $\gamma$, that is
the combinatorial length of the minimal edge-path $\hat\gamma$
containing $\gamma$ as a subpath.

We need the following quantitative version of the statement known as
``Bounded Cancellation Lemma'' that was originally proved in ~\cite{Cooper}.

\begin{propdfn}[Bounded Cancellation Constant]\cite[Lemma 3.1]{BFH97}\label{prop:bc}

Let $f:\Gamma\to\Gamma$ and $m$ be as Convention~\ref{conv:Gamma}. Put $C_f:=m||f||$. 

Consider  any tight
path $\gamma=\alpha\beta$ in $\Gamma$ (where the endpoints of
$\alpha,\beta$ are not assumed to be vertices). Let $\alpha'$ be the
tightened form of $f(\alpha)$, and let $\beta'$ be the tightened form
of $f(\beta)$. 

Then for the maximal terminal segment $\gamma$
of $\alpha'$, that cancels with an initial segment of $\beta'$ when the
path $\alpha'\beta'$ is tightened, we have $|\gamma|\le C_f$.
\end{propdfn}

\begin{lem}\label{lem:2}
Let $g:\Delta\to\Delta$ be an expanding graph map, where $\Delta$ is a
finite graph with $s\ge 1$ topological edges.Then for every
$e\in E\Delta$ we have $|g^s(e)|\ge 2$. 
\end{lem}

\begin{proof}
Let $e\in E\Delta$. If $|g(e)|\ge 2$ then $|g^s(e)|\ge |g(e)|\ge 2$,
and we are done. Thus assume that $|g(e)|=1$. Since by assumption $g$
is expanding, $|g^n(e)|\to\infty$ as $n\to\infty$.
Let $k\ge 1$ be the
maximal integer such that $|g^k(e)|=1$. Thus for $i=1,\dots, k$
$e_i:=g^i(e)$ is a single edge, and $|g^{k+1}(e)|\ge 2$. Also put $e_0:=e$.
The underlying topological edges of the edges $e_0,e_1,\dots, e_k$
must be distinct since otherwise  $|g^n(e)|=1$ for all $n\ge 1$,
yielding a contradiction. Therefore $k+1\le s$. Hence $|g^s(e)|\ge
|g^{k+1}(e)|\ge 2$, as required.

\end{proof}

\begin{cor}\label{cor:INP-len}
Let $\eta$ be an INP for $f$, so that by Proposition~\ref{prop:key},
$\eta$ has the form  $\eta=\alpha\beta^{-1}$, where $\alpha,\beta$ are
nontrivial legal paths with $v=t(\alpha)=t(\beta)\in V\Gamma$ such that
the turn at $v$ between $\alpha$ and $\beta$ is illegal.

Then $|\alpha|, |\beta|\le m||f||^m+4$.

\end{cor}
\begin{proof}

Suppose that $|\alpha|>m||f||^m+4$, that is $|\alpha|\ge m||f||^m+5$. Thus $\alpha$
contains as a subpath an edge-path of combinatorial length $|\alpha|-2\ge
m||f||^m+3$. Then by Lemma~\ref{lem:2}, $|f^m(\alpha)|\ge 2(|\alpha|-2)$.

 Note that
$f^m:\Gamma\to\Gamma$ is a train-track map with $||f^m||\le
||f||^m$. Hence, by Proposition-Definition~\ref{prop:bc}, $C_{f^m}\le
m||f||^m$. The paths $f^m(\alpha)$ and $f^m(\beta)$ are already tight
and legal.
Therefore $f^m(\alpha)$ contains an edge-subpath $\tau$ that survives
after tightening $f^m(\alpha)f^m(\beta^{-1})$ and such that 
\[
|\tau|\ge 2(|\alpha|-2)-C_{f^m}\ge 2|\alpha|-4-m||f||^m\ge |\alpha|+1
\]
where the last inequality holds because we assumed that $|\alpha|\ge
m||f||^m+5$.

However, $\eta=\alpha\beta^{-1}$ is a Nielsen path for
$f^m$, with $\alpha,\beta$ legal and the turn between $\alpha$ and
$\beta$ illegal. Therefore, after tightening $f^m(\alpha)f^m(\beta^{-1})$ the
portion of $f^m(\alpha)$ that does not cancel must be exactly its initial
segment $\alpha$ which has combinatorial length  $|\alpha|$,
yielding a contradiction. Thus $|\alpha|\le m||f||^m+4$. A symmetric
argument shows that $|\beta|\le m||f||^m+4$.
\end{proof}

\begin{conv}
Note that the expanding assumption on $f$ implies that if $e\in
E\Gamma$ and $k\ge 1$ then $f^k(e)\ne e^{\pm 1}$. Let $Fix(f):=\{x\in
\Gamma| f(x)=x\}$ be the fixed set of $f$.

If $f(e)=e_1\dots e_n$ and $e_i=e^{\pm 1}$ for some $1<i<n$ or if
$f(e_i)=e^{-1}$ for $i\in \{1,n\}$ then we get a fixed point of $f$ in
the interior of the subinterval of $e$ that gets mapped to $e_i$.
Moreover, every fixed point of $f$, contained in the interior of some
edge, arises in this fashion. (Here we are using our PL assumption on
$f$). Thus for every $e\in E\Gamma$ there are at most $|f(e)|$ fixed
points of $f$ in the interior of $e$ and at most $|f(e)|+1$ fixed
points in the closure of $e$. Therefore $\#(Fix(f))\le m (||f||+1)\le 2m||f||$.

Let $\Gamma'$ be the graph obtained from $\Gamma$ by subdividing along
all elements of $Fix(f)$ which are not already vertices of
$\Gamma$. We denote by $f':\Gamma'\to\Gamma'$ the map $f$ considered
as a graph map $\Gamma'\to\Gamma'$. 

Every path in $\Gamma$ is a path in $\Gamma'$. Sometimes, in order to
avoid confusion, for a path $\alpha$ in $\Gamma$, we will denote its
combinatorial length (in the sense defined in the beginning of
Section~\ref{sec:b}) in $\Gamma$ by $|\alpha|_\Gamma$ and we  will denote its
combinatorial length in $\Gamma'$ by $|\alpha|_{\Gamma'}$.

\end{conv}

The following is a direct consequence of the definition of $f'$ and
$\Gamma'$:
\begin{lem}\label{lem:Gamma'}
The following hold:
\begin{enumerate}
\item Each edge of $\Gamma$ got subdivided into $\le |f(e)|_\Gamma$
  edges in $\Gamma'$. Hence the graph $\Gamma'$ and the map
  $f':\Gamma'\to \Gamma'$ can be computed in time $O\left(||f|| m(\log
  m+\log ||f||)\right)$.  [The $\log m+\log ||f||$-factor arises since $\Gamma'$ has $\le
  m||f||$ topological edges and, when describing $f'$, we first need
  to index these edges by binary integers.]  
\item For every path $\alpha$ in $\Gamma$ we have
$|\alpha|_\Gamma\le |\alpha|_{\Gamma'}\le ||f||\cdot
|\alpha|_\Gamma$.
\item  We have $||f'||\le ||f||^2$.
\end{enumerate}
\end{lem}

\begin{rem}\label{rem:comp}
If the number $m$ of topological edges of $\Gamma$ is considered a constant, and $\alpha$ is an edge-path in $\Gamma$ then for $t\ge 1$ we have $|f^t(\alpha)|_\Gamma\le ||f||^t |\alpha|_\Gamma$ and the path $f^t(\alpha)$ can be computed in $O( ||f||^t |\alpha|_\Gamma)$ steps. However, it is often more efficient to use a different computational bound in this situation, since $|f^t(\alpha)|_\Gamma$ may be much smaller than $||f||^t |\alpha|_\Gamma$. Note that  $f(\alpha)$ can be computed in $O(|f(\alpha)|_\Gamma)$ steps and hence $f^t(\alpha)$ can be computed in $O(t|f^t(\alpha)|_\Gamma)$ steps (if $m$ is viewed as a constant). If $m$ is a variable parameter, then in these computations we first need to index the edges of $\Gamma$ as $e_1,\dots, e_m$, with the subscript $i$ of $e_i$ written in binary form. Therefore computing $f(\alpha)$ can be done in $O(|f(\alpha)|_\Gamma\log m)$ steps and hence $f^t(\alpha)$ can be computed in $O(t|f^t(\alpha)|_\Gamma\log m)$ steps. The graph $\Gamma'$ has $m'\le m||f||$ topological edges. Therefore, by similar reasoning, given an edge-path $\alpha$ in $\Gamma'$, we can compute $(f')^t(\alpha)$ in $O\left(t|(f')^t(\alpha)|_{\Gamma'})(\log m+\log||f||)\right)$ steps.  We will use these bounds in the computations below.
\end{rem}

\begin{defn}[Eigenrays]
Let $d$ be a fixed direction in $\Gamma$ given by a nondegenerate
segment $\tau$ in $\Gamma$ starting at a fixed point $x\in \Gamma$ and
such that $f(\tau)$ has $\tau$ as an initial segment. Since $f$
is expanding, the PL assumption on $f$ implies that $\tau$ is a proper
initial segment of $f(\tau)$ and that $|f^k(\tau)|\to\infty$ as
$k\to\infty$.  Thus we can form a reduced legal
path $\rho_d$ in $\Gamma$ such that for every $k\ge 1$
$f^k(\tau)$ is an initial segment of $\rho_d$.

This path $\rho_d$ is called the \emph{eigenray} of $f$ determined by
$d$. (It is not hard to check that this definition of $\rho_d$ does
not depend on the choice of $\tau$ as above).
\end{defn}

Note that in the above situation $f(\rho_d)=\rho_d$. Since $f=f'$ as functions, and since the graphs
$\Gamma$ and $\Gamma'$ have the same set of fixed points and the same
set of fixed directions at fixed points, in the above definition it
does not matter whether we use $f:\Gamma\to\Gamma$ or $f':\Gamma'\to\Gamma'$.  Also, the maps $f$ and $f'$ have exactly the
same collection of INPs.

Now part (2) of Proposition~\ref{prop:key} implies:

\begin{lem}
Let $\eta=\alpha\beta^{-1}$ be an INP in $\Gamma$, where
$\alpha,\beta$ are as in part (2) of Proposition~\ref{prop:key}. Let $d$ be the direction at $o(\alpha)$ given by $\alpha$ and let
  $d'$ be the direction at $o(\beta)$ given by $\beta$.

Then:
\begin{enumerate}
\item There is a legal path $\gamma$ in $\Gamma$ such that
$f(\alpha)=\alpha\gamma$ and $f(\beta)=\beta\gamma$. 
\item The directions $d$ and
  $d'$ are fixed by $f$.
\item The path $\alpha$ is an initial segment of the eigenray
  $\rho_d$, and the path  $\beta$ is an initial segment of the eigenray
  $\rho_{d'}$.
\end{enumerate}
\end{lem}

\begin{rem}\label{rem:comput-INP}
Let $f:\Gamma\to\Gamma$ be a standard expanding irreducible train
track map. We should explain now what we mean by ``computing'' an INP
or a pINP for $f$. Recall that every pINP $\eta$ for $f$ has the form $\eta=\alpha\beta^{-1}$,  where
$\alpha,\beta$ are as in part (2) of Proposition~\ref{prop:key}. Let
$k\ge 1$ be some integer such that $\eta$ is an INP for $f^k$.
In particular $\alpha$ and $\beta$ begin at $f$-periodic points fixed
by $f^k$ and meet at nondegenerate illegal turn. 

Then either $\alpha$ is an edge-path in $\Gamma$ or $\alpha$ begins at
an interior point $x$ of an edge $e$ of $\Gamma$. In the first case
by ``computing'' $\alpha$ we just mean recording the corresponding
edge-path.

Suppose now the latter happens. Then $\alpha=e'\alpha'$ where $e'$ is
a proper terminal segment of $e$ starting with $x$, and where
$\alpha'$ is a (possibly empty) edge-path in $\Gamma$. Recall that $f^k$ is not necessarily a
standard graph map. Still, $f^k(e)$ is, combinatorially, an edge-path
in $\Gamma$ of the form $e_1,e_2,\dots,e_r$. The edge $e$ is
subdivided into $r$ consecutive closed intervals $J_1,J_2,\dots J_r$,
with $f^k$ taking the interior of $J_i$ to the open edge $e_i$ for
$i=1,\dots, r$. There is a unique index $q$ such that $x$ belongs to
the interior of $J_q$; moreover in this case $e_q=e^{\pm 1}$.  Recall
that in our graph $\Gamma$ the edge $e$ is equipped  with the
characteristic map $\xi_e:e\to (0,1)$. Applying $\xi_e$ to the
subdivision $J_1,J_2,\dots J_r$ of $e$, we get a subdivision
\[
0=s_0< s_1 <\dots < s_r=1  
\]
of $[0,1]$, with $\xi_e(x)\in (s_{q-1},s_q)$.

In this case, by ``computing'' $e'$ we mean finding $k$, the path
$e_1,e_2,\dots,e_r$, the subdivision $0=s_0< s_1 <\dots < s_r=1$ and
the (rational) number $\xi_e(x)\in (s_{q-1},s_q)$. 
Then ``computing'' $\alpha$ means computing $e'$ together with
computing the edge-path $\alpha'$ in $\Gamma$.

The meaning of  ``computing'' $\beta$ is defined similarly.

Finally, by computing an INP
or a pINP $\eta$ for $f$ we mean finding the decomposition of
$\eta$ in the form $\eta=\alpha\beta^{-1}$ as above and computing each
of $\alpha$ and $\beta$ in the above sense.

Note also that for practical computations with INPs and pINPs knowing
the combinatorial structure of $\alpha$ and $\beta$ is often
sufficient. Thus, when describing $e'$ above, it is often sufficient
to specify $k$, the combinatorial path $e_1,\dots, e_k$ and the index
$q$ such that $x$ occurs in the interior of the interval $J_q$ being
mapped by $f^k$ to $e_q=e^{\pm 1}$. This point of view is taken in the
train track computational software package developed by Thierry
Coulbois~\cite{Col}.
\end{rem}

We are now ready to prove Theorem~\ref{thm:alg-inp}.

\begin{proof}[Proof of Theorem~\ref{thm:alg-inp}]

First we compute the set $Fix(f)$ (using the PL definition of $f$),
and construct the graph $\Gamma'$ and compute the map
$f':\Gamma'\to\Gamma'$.  By Lemma~\ref{lem:Gamma'}, this task can be
done in $O(m||f||)$ steps. We then find all the fixed directions at
vertices of $\Gamma'$ by checking which edges $e'\in E\Gamma'$ have the
property that $f'(e')$ starts with $e'$. Note that the number of
vertices of $\Gamma'$ and of directions at vertices of $\Gamma'$ is $O(m||f||)$, and therefore this task can also
be done is  in $O(m||f||)$ steps. Put $m'$ to be the number of
non-oriented edges in $\Gamma'$, so that $m'=O(m||f||)$.

For each edge $e'$ of $\Gamma'$ defining a fixed direction of $f'$ we
start computing the eigenray $\rho_{e'}$ by iterating $f$ on $e'$. We
iterate until we have found $k\ge 1$ such that $|f^{km'}(e')|_{\Gamma'}\ge
||f||(m||f||^m+5)$, so that $|f^{km'}(e')|_{\Gamma}\ge
(m||f||^m+5)$. By Lemma~\ref{lem:2} applied to the map $f'$, we know
that $f^{m'}$ sends every edge of $\Gamma'$ to an edge-path of length
$\ge 2$ in $\Gamma'$. Thus $|f^{tm'}(e')|_{\Gamma}|\ge 2^t$ for $t\ge
1$. Therefore we can find $k\ge 1$ such that  $|f^{km'}(e')|_{\Gamma'}\ge
||f||(m||f||^m+5)$ with $k=O(m\log||f||)$. Since $m'=O(m||f||)$, we
have $km'=O(m^2||f||\log ||f||)$. Therefore, by Remark~\ref{rem:comp}, for a given $e'$ computing
the path $\tau(e'):=f^{km'}(e')$ in $\Gamma'$ can be done in 
$O(m^2||f||\log ||f|| ||f||(m||f||^m+5)(\log m+\log||f||))=O(m^3||f||^{m+2}\log^2||f||\log m)$ steps.

As there are at most $O(m||f||)$ edges in $\Gamma'$ giving fixed
directions for $f'$, we find all such edges $e_1',\dots, e_q'$ (with
$q=O(m||f||)$), and for each of them compute the path $\tau(e_i')$ as
above. This requires at most   $O(m^4||f||^{m+3}\log^2||f||\log m)$ steps.

Now enumerate pairs $(\alpha,\beta)$ where $\alpha$ is an initial
segment of some $\tau(e_i')$ and where $\beta$ is an initial
segment of some $\tau(e_j')$.  Since $q=O(m||f||)$ and  $|\tau(e')|_{\Gamma'}=O(m||f||^{m+1})$, the number of such pairs $(\alpha,\beta)$ is at most $O(m^4||f||^{2m+4})$ and they can be enumerated in time $O(m^4||f||^{2m+4}\log m\log||f||)$.

For each such pair $(\alpha,\beta)$ we check whether $\eta=\alpha\beta^{-1}$ is a path in $\Gamma'$ and if this path is tight. If yes, we compute $f(\eta)$ and then compute the tightened form of $f(\eta)$. For a given pair $(\alpha,\beta)$, this check can be done in $O(m||f||^{m+2}\log m\log ||f||)$ steps, and doing this for all pairs $(\alpha,\beta)$ requires at most $O(m^5||f||^{3m+6}\log m\log ||f||)$ steps. If the tightened form of $f(\eta)$ is $\eta$, then by Proposition~\ref{prop:key} $\eta$ is an INP for $f$ (note that in this case the turn between $\alpha$ and $\beta$ in $\eta$ is automatically illegal since $f$ is expanding), and, also by Proposition~\ref{prop:key}, every INP for $f$ arises in this way.
Thus running this check for every pair $(\alpha,\beta)$ as above computes all the INPs for $f$.

Summing up the above computations, the total running time of this
process can be bounded above by $O(m^5||f||^{3m+6}\log m\log ||f||)$ steps. 
\end{proof}

Theorem~\ref{thm:alg-inp} can be used to establish the following more
general result:

\begin{thm}\label{thm:turner}
Let $N\ge 2$ be fixed.

The exists a deterministic algorithm $\mathfrak A'$ with the following property.
Given an expanding standard train track map $f:\Gamma\to\Gamma$ (where $\Gamma$
is a finite connected graph with all vertices of degree $\ge 3$ and
with $\pi_1(\Gamma)$ free of rank $N$), the
algorithm $\mathfrak A'$ determines whether or not $f$ has any
periodic INPs and if yes, computes all the periodic INPs for $f$. 

The algorithm then constructs the graph $S(f)$.

The algorithm terminates in polynomial time in $||f||$. 
\end{thm}

\begin{proof}
As noted in Remark~\ref{rem:FH},  Feighn and Handel showed in
\cite[Corollary 3.14]{FH} that there exists a computable power
$t=t(N)$ such that for $f^t:\Gamma\to\Gamma$ a path $\eta$ is a periodic indivisible
Nielsen path for $f$ if and only if $\eta$ is an INP for $f^t$. Note
that $||f^t||\le ||f||^t$. 

Thus we want to first replace $f$ by $f^t$ and then apply
Theorem~\ref{thm:alg-inp} to $f^t$. However, a technical complication
arises in that $f^t:\Gamma\to\Gamma$ is not necessarily a standard
graph map and thus Theorem~\ref{thm:alg-inp} technically does not
directly apply to $f^t$. 

However, there is a unique standard graph map $g:\Gamma\to \Gamma$ which is
isotopic to $f^t$ relative $VG$.

There is a also polynomial time computable ``reparameterization''
map $h:G\to G$, such that
$h|_{VG}=ID_{VG}$, such that the restriction of $h$ to each edge $e$
of $G$ is a piecewise-linear orientation-preserving homeomorphism from
$e$ to $e$, and such that  $h$ conjugates $f^t$ to $g$. More
precisely, we have $f^t=h\circ g\circ h^{-1}$. Then whenever $\eta'$
is an INP for $g$ then $h\eta'$ is an INP for $f^t$, and, moreover,
all INPs for $f^t$ arise in this way.

Thus to find pINPs for $f$, we use  Theorem~\ref{thm:alg-inp} to
compute all INPs (if any) $\eta_1',\dots,  \eta_k'$ for $g$. Then
$h\eta_1',\dots,  h\eta_k'$ are all the pINPs for $f$. 
\end{proof}

\begin{rem}\label{rem:vol}
Let $N\ge 3$ and $f:\Gamma\to\Gamma$ be as in
Theorem~\ref{thm:turner}. Then $\Gamma$ has $\le 3N$ topological
edges, $\le 6N$ oriented edges and $\le 36N^2$ illegal turns at
vertices. Therefore there are at most $36N^2$ pINPs distinct in
$\Gamma$ (up to inversion). Let $\eta_1,\eta_2,\dots, \eta_s$ be these
pINPs, so that $0\le s\le 36N^2$.

Recall that, as noted in Remark~\ref{rem:FH}, there exists a uniform
power $b=b(N)$ such that every pINP of $f$ is an INP of $f^b$. 
Corollary~\ref{cor:INP-len} implies that if $\eta$ is an INP for $f^b$
then $|\eta|\le 6N||f^b||^{3N}+8$. Thus for
each $0\le i\le s$ we have $|\eta_i|\le 6N||f||^{3bN}+8$. 
\end{rem}

\begin{prop}\label{prop:pa}
Let an integer $N\ge 2$ be fixed.
There exists an algorithm that, given an expanding irreducible
standard train track map $f:\Gamma\to\Gamma$ with $\pi_1(\Gamma)\cong F_N$, decides whether or not $\phi=f_\#$ is primitively atoroidal. This algorithm terminates in at most polynomial time in $||f||$. 
\end{prop}
\begin{proof}
We first algorithmically construct the graph $S(f)$ using Theorem~\ref{thm:turner}; this can be done in polynomial time in $||f||$. If $b_1(S(f)) =0$, then $\phi$ is atoroidal and, in particular, primitively atoroidal. Thus assume that $b_1(S(f)) \ge 1$.

We list all the non-contractible connected components $Q_1,\dots, Q_k$ of $S(f)$.  In view of Remark~\ref{rem:vol}, the
number $k$ is bounded above by  $36N^2$, and the total sum of the
lengths of the $\mu$-labels of edges of $S(f)$ is bounded above by a
polynomial function in terms of $||f||$. Then, using the Stallings
folding algorithm, for each $Q_i$ we find a finite generating set of
the subgroup $U_i=\mu(\pi_1 Q_i)\le \pi_1(\Gamma)$.  Note that the sum of the lengths of the generators of $U_1,\dots, U_k$, expressed as loops in $\Gamma$, is bounded by a polynomial in $||f||$. 
Again, this step requires at most polynomial time in $||f||$.

Proposition~\ref{prop:mu} implies that a nontrivial conjugacy class $[g]$ in $\pi_1(\Gamma)$ is $\phi$-periodic if and only if $g$ is conjugate in $\pi_1(\Gamma)$ to an element of $U_i$ for some $1\le i\le k$.

Now Corollary~13.3 in \cite{FH} implies that either for $i=1,\dots, k$ every element of $U_i$ is trivial in $H_1(F_N,\mathbb Z/2\mathbb Z)$, or $\phi=f_\#$ is not primitively atoroidal (and these two cases are mutually exclusive).  We then compute the images $\overline{U_1}, \dots \overline{U_k}$ of $U_1,\dots, U_k$ in $H_1(F_N,\mathbb Z/2\mathbb Z)$ and check if we have $\overline{U_i}=\{0\}$ for $i=1,\dots, k$. [It is clear that this check can also be done in polynomial time.] If yes, then $\phi=f_\#$ is primitively atoroidal, and if not then $\phi=f_\#$ is not primitively atoroidal.
\end{proof}

\begin{rem}
 Our original proof of Proposition~\ref{prop:pa} was different and used the result of Clifford-Goldstein~\cite{CG}, or alternatively, of Dicks~\cite{D}, to decide whether any of the subgroups $U_1,\dots, U_k$ of $F_N$ contains a primitive element of $F_N$. The referee then pointed to us a simpler argument, presented above, using Corollary~13.3 of \cite{FH}.

Feighn and Handel~\cite[Corollary 3.14]{FH} obtain a general algorithm for deciding if an element $\phi\in \Out(F_N)$ is primitively
atoroidal. However, their approach relies on heavy duty machinery of CT train tracks and needs, as an essential ingredient, the main result of \cite{FH} about algorithmically constructing CTs.
That result does not come with any complexity estimate and it remains to be seen if reasonable complexity estimates can be obtained there. Therefore the general algorithm given in \cite[Corollary 3.14]{FH} also does not, for the time being, have a complexity estimate.

\end{rem}

\section{Refined full irreducibility testing}

In this section we provide a refinement of the algorithm from \cite{Ka14} for
detecting full irreducibility. 

\begin{prop}\label{prop-wh}
Let $N\ge 2$ be fixed.

There exists an algorithm that, given a finite connected
graph where every vertex has degree $\ge 3$ and $\pi_1(\Gamma)\cong
F_N$, and given a standard expanding
irreducible train track map $f:\Gamma\to\Gamma$, computes, in
polynomial time, in terms of $||f||$, the Whitehead graph $Wh_\Gamma(v,f)$
for every vertex $v\in V\Gamma$. 
\end{prop}

\begin{proof}

Under the assumptions made on $\Gamma$ there are at most $3N$
topological edges and at most $6N$ oriented edges. Thus there are at
most $36N^2$ turns in $\Gamma$. 

Let us enumerate all the edges $e_1,\dots, e_k$ of $\Gamma$ (so that
$k\le 6N$).

We put $\mathcal T_0=\varnothing$, and $\mathcal T_1$ to be the set of
turns contained in $f(e_1),\dots, f(e_k)$. Computing  $\mathcal T_1$  requires at most
$O(||f||)$ time. 

Then, inductively, given a set of turns $\mathcal T_i$, we define
$\mathcal T_{i+1}:=\mathcal T_i\cup Df(\mathcal T_i)$ where $Df$ is the
derivative map of $f$. Thus 
\[
\varnothing=\mathcal T_0\subseteq \mathcal T_1 \subseteq \mathcal T_2
\subseteq \dots \subseteq \mathcal T
\]
where $\mathcal T$ is the set of all turns in $\Gamma$.

Note that, by construction, once there is some $i$ such that $\mathcal
T_i=\mathcal T_{i+1}$ then $\mathcal T_j=\mathcal T_i$ for all $j\ge i$.

Since $\#\mathcal T\le 36N^2$, this chain terminates in at most
$36N^2$ steps with some $q\le 36N^2$ such that $\mathcal
T_q=\mathcal T_{q+1}$. Summing up the times, we see that the set $T_q$ can be
computed in polynomial time in $||f||$. The set $T_q$ is the set of
all turns ``taken'' by the  train track map $f$. Using this set we can
now construct all the Whitehead graphs $Wh_\Gamma(v,f)$ of the
vertices of $\Gamma$, using Definition~\ref{defn:wh-lam}. Note that the number of vertices of $\Gamma$ is
$\le 6N$, and that for every vertex of $\Gamma$ its degree is $\le
2N$. This gives us constant (in terms of $||f||)$ bounds on the number
of Whitehead graphs we need to construct, and on the sizes of each of
these graphs. Therefore the total computational bound remains
polynomial in terms of $||f||$.

\end{proof}

\begin{thm}\label{thm:alg}
Let $N\ge 2$ be fixed.

\begin{enumerate}

\item There exists an algorithm that, given a standard topological
  representative $f':\Gamma'\to \Gamma'$ of some $\phi\in \Out(F_N)$,
  such that every vertex in $\Gamma$ has degree $\ge 3$, decides
  whether or not $\phi$ is fully irreducible.  The algorithm terminates in polynomial time in terms of $||f'||$.

\item Let $A=\{a_1,\dots,a_N\}$ be a fixed free basis of $F_N$. There exists an algorithm, that given $\Phi\in \Aut(F_N)$ (where $\Phi$ is given as an $N$-tuple of freely reduced words over $A^{\pm 1}$, $(\Phi(a_1),\dots, \Phi(a_N))$), decides, in polynomial time in $|\Phi|_A$, whether or not $\Phi$ is fully irreducible.

\item Let $S$ be a finite generating set for $\Out(F_N)$. There exists
  an algorithm that, given a word $w$ of length
  $n$ over $S^{\pm 1}$, decides whether or not the element $\phi$ of
  $\Out(F_N)$ represented by $w$ is fully irreducible. This algorithm terminates in at most exponential time in terms of the length
of the word $w$.

\end{enumerate}

\end{thm} 
\begin{proof}

(1) We first run the algorithm of Bestvina-Handel for trying to find an
irreducible standard train track representative of  $\phi$. By
Theorem~\ref{thm:BH92}, this
algorithm always terminates in polynomial time in $||f'||$, and either finds a reducible topological
representative of $\phi$ (in which case we conclude that $\phi$ is not
fully irreducible), or it finds an irreducible standard train track representative $f:\Gamma\to\Gamma$ of  $\phi$. 
Assume that the latter has occurred. In this case, by Remark~\ref{rem:poly}, we also have that
$||f||$ is bounded by a polynomial function in terms of $||f'||$. We then check whether $M(f)$ is a permutation matrix. If yes, then $\phi$ has finite order in $\Out(F_N)$ and thus is not fully irreducible. If not, then $f$ is an expanding irreducible train track representative of $\phi$.

We then use the algorithm obtained in Proposition~\ref{prop:pa} to decide, in polynomial time in $||f'||$, whether or not $\phi$ is primitively atoroidal.  If $\phi$ is not primitively atoroidal, then some positive power of $\phi$ preserves the conjugacy class of a rank-1 free factor of $F_N$, and thus $\phi$ is not fully irreducible.

Suppose now that $\phi$ is primitively atoroidal. We then construct the Whitehead graphs $Wh_\Gamma(v,f)$
for every vertex $v\in V\Gamma$.  Note that by
Proposition~\ref{prop-wh} this step can be done in polynomial time in
terms of $||f||$. Moreover, since there are at most $6N$ vertices in
$\Gamma$, and for every $v\in V\Gamma$ the Whitehead graphs
$Wh_\Gamma(v,f)$ has $\le 2N$ vertices, we can also verify in
polynomial time in $||f||$ whether or not the Whitehead graphs
$Wh_\Gamma(v,f)$ are connected for all $v\in V\Gamma$

Theorem~\ref{thm:clean} now implies
that $\phi$ is fully irreducible if and only if all these Whitehead
graphs are connected.

Thus we have produced an algorithm which decides whether or not $\phi$
is fully irreducible.  Summing up the total time expended,  we see that the algorithm does terminate in polynomial time in $||f'||$.

(2) The statement of part (2) follows from part (1).
We take $\Gamma'$ to be an $N$-rose with petals marked by $a_1,\dots a_N$. Take $u_i=\Phi(a_i)$ for $i=1,\dots, N$.
Now construct a standard graph map $f':\Gamma'\to\Gamma'$ where the
$i$-th petal of $\Gamma'$ gets mapped according to the word $u_i$. Then $||f'||=|\Phi|_A$ and $f'$ is a standard topological representative of (the outer automorphism class of) $\phi$.
Applying the result of part (1) to $f'$ we obtain the desired conclusion.

(3) The statement of part (3) follows from part (1).

Again, we take $\Gamma'$ to be an $N$-rose. If $S=\{\sigma_1,\dots,\sigma_k\}$
and $F_N=F(a_1,\dots,a_N)$, we put $C=\max_{i=1}^N
|\sigma_i(a_i)|_A$. We first compute the freely reduced words
$u_1,\dots, u_N$ over $A=\{a_1,\dots,a_N\}$ where
$u_i=\phi(a_i)$. Since $\phi$ is represented by a word $w$ over
$S^{\pm 1}$, we have $|u_i|_A\le C^{|w|}$ for $i=1,\dots N$.

Then construct a standard graph map $f':\Gamma'\to\Gamma'$ where the
$i$-th petal of $\Gamma'$ gets mapped according to the word $u_i$.  Note that $||f'||\le C^{|w|}$.
Now applying the result of part (1) to $f'$ we obtain the desired
conclusion for (3).
\end{proof}

\appendix

\begin{absolutelynopagebreak}
\section{The complexity of the Bestvina-Handel algorithm}
\medskip
\centerline{\em by Mark C. Bell}
\medskip
\centerline{University of Illinois}
\centerline{\url{mcbell@illinois.edu}}
\bigskip

Recall that a \emph{topological representative} of an outer automorphism $\phi \in \Out(F_N)$ is a homotopy equivalence $f : \Gamma \to \Gamma$ taking vertices to vertices, where $\Gamma$ is a finite connected graph whose fundamental group has been identified with $F_N$ and where each vertex has degree at least three, such that $f$ induces $\phi$ on $\pi_1(\Gamma)$.
For ease of notation, we will say that a topological representative $f$ of $\phi$ is \emph{good} if it is either a train track map or is reducible.
\end{absolutelynopagebreak}

In a seminal 1992 paper~\cite{BH92} Bestvina and Handel gave an algorithm to improve any topological representative $f' : \Gamma' \to \Gamma'$ of $\phi \in \Out(F_N)$ to a good one $f : \Gamma \to \Gamma$.
Hence showing that every irreducible outer automorphism has a train track representative \cite[Theorem~1.7]{BH92}.
Here we provide an explicit polynomial bound on the running time of their algorithm:

\begin{thm}[Complexity of the Bestvina-Handel algorithm]
\label{thm:BH92}
Fix an integer $N \ge 2$.
Given a topological representative $f'$ of $\phi \in \Out(F_N)$, the Bestvina--Handel algorithm terminates in polynomial time in $||f'||$.
\end{thm}

\begin{cor}
Fix an integer $N \geq 2$ and let $S$ be a finite generating set of $\Out(F_N)$.
There is an algorithm that, given a word $w$ over $S^{\pm 1}$ corresponding to $\phi \in \Out(F_N)$, finds a good topological representative of $\phi$.
Moreover, this algorithm terminates in exponential time in $|w|$.
\end{cor}

\begin{proof}
Fix a free basis $A = \{a_1, \dots, a_N\}$ of $F_N$.
We first compute the freely reduced words $u_1, \dots, u_N$ over $A$ where $u_i = \phi(a_i)$.
Since $\phi$ is given by a word $w$ over $S^{\pm 1}$, we have that $|u_i|_A \leq C^{|w|}$ where
\[ C \defeq \max_{\sigma \in S} \max_{a_i \in A} |\sigma(a_i)|_A \]
is a constant that depends only on $A$ and $S$.

Now let $\Gamma'$ to be the $N$--rose with the petals marked by $a_1, \dots, a_N$.
Let $f' : \Gamma' \to \Gamma'$ be the \emph{standard} topological representative of $\phi$ where the $i$-th petal of $\Gamma'$ gets mapped according to the word $u_i$.
Note that the bounds on $|u_i|_A$ imply that $||f'|| \leq C^{|w|}$ and so the result now follows by applying Theorem~\ref{thm:BH92} to $f'$.
\end{proof}

The Bestvina--Handel algorithm works by producing a sequence of topological representatives $f_0, \ldots f_K$, starting with $f_0 = f'$.
The representative $f_{i+1}$ is computed from $f_i$ by applying one of four basic operations \cite[Section~1]{BH92}.
These operations ensure each topological representative is better than the last: the stretch factor of $f_{i+1}$ is strictly less than the stretch factor of $f_i$.
Bestvina and Handel showed that when this procedure terminates --- when no such operation can improve the topological representative --- the map $f_K$ is good \cite[Theorem~1.7]{BH92}.

Examination of the different possible transformations shows that $f_{i+1}$ can be computed from $f_i$ in $\poly(||f_i||)$ time and that $||f_{i+1}|| \leq \poly(||f_i||)$.
However these bounds alone are not sufficient to conclude that the Bestvina--Handel algorithm runs in polynomial time.
Thus we prove Theorem~\ref{thm:BH92} by showing that $K$ and $||f_i||$ are bounded by polynomial functions of $||f'||$.
We achieve this by considering the transition matrices $M_i$ associated to $f_i$.

Recall that a matrix $M \in \NN_0^{n \times n}$ is \emph{irreducible} if for each $i$ and $j$ there exists $k$ such that $(M^k)_{ij} \neq 0$ \cite[Page~671]{Meyer}.
We may think of a matrix as the adjacency matrix of a directed graph.
Doing so we obtain that a matrix is irreducible if and only if its corresponding digraph is strongly connected \cite[Page~671]{Meyer}.
From this we deduce that we may strengthen the definition of a matrix being irreducible to also include that $1 \leq k \leq n$.

An irreducible matrix $M$ has an associated \emph{Perron--Frobenious eigenvalue} $\lamPF(M)$.
This is equal to its \emph{spectral radius} $\rho(M) \geq 1$ and is the largest eigenvalue of $M$ in absolute value \cite[Page~673]{Meyer}.

For a matrix $M = (m_{ij})$, let $||M|| \defeq \max(m_{ij})$ denote the maximum of the entries of $M$.
By the Gershgorin circle theorem \cite{Gerschgorin}, we immediately obtain an upper bound on $\lamPF(M)$:

\begin{lem}
\label{lem:upper_bound}
If $M \in \NN_0^{n \times n}$ is irreducible then
\[ \lamPF(M) \leq n \cdot ||M||. \inlineQED \]
\end{lem}

The main tool powering this section is a corresponding lower bound.

\begin{prop}
\label{prop:lower_bound}
If $(m_{ij}) = M \in \NN_0^{n \times n}$ is an irreducible matrix then
\[ ||M|| \leq \sum_{i,j=1}^n m_{ij} \leq n \lambda^{n+1} \]
where $\lambda \defeq \lamPF(M)$.
\end{prop}

\begin{proof}
Let $(x_1 \; \cdots \; x_n)^T \neq 0$ be a real eigenvector of $M$ corresponding to $\lambda$ such that $x_i \geq 0$ \cite[Equation~8.3.2]{Meyer}.
For ease of notation let $m_{ij}^k \defeq (M^k)_{ij}$.
Now note that as $M$ is irreducible for each $i$ and $j$ there is a $1 \leq k \leq n$ such that $m_{ij}^k \geq 1$.
From this we observe that for each $i$ we have that
\begin{eqnarray*}
n \lambda^{n} x_i &\geq& (\lambda^n + \cdots + \lambda) x_i \\  
 &=& \sum_{k=1}^n \sum_{j=1}^n m_{ij}^k x_j \\
 &\geq& x_1 + \cdots + x_n.
\end{eqnarray*}
Thus
\begin{eqnarray*}
\lambda (x_1 + \cdots + x_n) &=& \sum_{i=1}^n \lambda x_i \\
 &=& \sum_{i,j=1}^n m_{ij} x_j \\
 &\geq& \frac{x_1 + \cdots + x_n}{n \lambda^{n}} \sum_{i,j=1}^n m_{ij}.
\end{eqnarray*}
By cancelling the $x_1 + \cdots + x_n \neq 0$ term on each side and rearranging this inequality we obtain
\[ \sum_{i,j=1}^n m_{ij} \leq n \lambda^{n+1} \]
as required.
The other inequality follows trivially.
\end{proof}

Now let $\PF(n, \lambda)$ denote the set of $m \times m$ irreducible matrices, where $m \leq n$, with entries in $\NN_0$ and Perron--Frobenious eigenvalue at most $\lambda$.
As an immediate consequence of Proposition~\ref{prop:lower_bound} we obtain the following bound on the cardinality of $\PF(n, \lambda)$:

\begin{cor}
\label{cor:PF_bound}
There are at most $n (n \lambda^{n + 1} + 1)^{n^2}$ matrices in $\PF(n, \lambda)$.
In particular, when $n$ is fixed this bound is polynomial in $\lambda$. \qed
\end{cor}

We now have the tools needed to prove the main theorem:

\begin{proof}[Proof of Theorem~\ref{thm:BH92}]
Let $f_0, \ldots, f_K$ be the sequence of topological representatives produced by the Bestvina--Handel algorithm.
Let $M_i$ be the transition matrix of $f_i$.
Note that $M_0$ is the transition matrix of $f_0 = f'$.
Hence if $M_0$ is reducible then the algorithm terminates immediately and there is nothing to check as $f = f'$.

Otherwise, $M_0$ is irreducible and so
\[ \lamPF(M_0) \leq N ||M_0|| \leq N ||f'|| \]
by Lemma~\ref{lem:upper_bound}.

These matrices have size at most $R \defeq 3N - 3$ and, while it is possible that $M_K$ is reducible, the matrices $M_0, \ldots, M_{K-1}$ are irreducible.
Now $\lamPF(M_{i}) < \lamPF(M_{i-1})$ for each $0 < i < K$.
Thus $M_0, \ldots, M_{K-1} \in \PF(R, N ||f'||)$ and these matrices are all pairwise distinct.
Hence \[ K \leq |\PF(R, N ||f'||)| + 1 \] which is at most a polynomial function of $||f'||$ by Corollary~\ref{cor:PF_bound} as $N$ is fixed.

Furthermore, for each $0 \leq i < K$ the fact that $M_i \in \PF(R, N ||f'||)$ implies that $||M_i|| \leq R (N ||f'||)^{R + 1}$ by Proposition~\ref{prop:lower_bound}.
This shows that for the corresponding topological representatives \[ ||f_i|| \leq R^2 (N ||f'||)^{R + 1} \] which again is a polynomial function of $||f'||$ as $N$ is fixed.

Finally, as noted above \[ ||f_{K}||\leq \poly(||f_{K-1}||) \leq \poly(R^2 (N ||f'||)^{R + 1}). \]
Hence the Bestvina--Handel algorithm terminates in polynomial time in terms of $||f'||$.
\end{proof}

\begin{rem}
\label{rem:poly}
Suppose that $f'$ is a topological representative of $\phi \in \Out(F_N)$ and that $f$ is the good topological representative of $\phi$ produced by the Bestvina--Handel algorithm starting from $f'$.
We highlight that the proof of Theorem~\ref{thm:BH92} also shows that, when $N$ is fixed, $||f||$ is bounded above by a polynomial function of $||f'||$.
\end{rem}

\end{document}